\documentclass[12pt]{article}

\usepackage[margin=1in]{geometry}
\usepackage{amsmath,amssymb,amsthm,mathtools}
\usepackage{enumitem}
\usepackage{graphicx}
\usepackage{booktabs}
\usepackage{tabularx}
\usepackage{natbib}
\usepackage[colorlinks=true,linkcolor=blue,citecolor=blue,urlcolor=blue]{hyperref}
\usepackage{microtype}

\newtheorem{definition}{Definition}

\newtheorem{theorem}{Theorem}
\newtheorem{proposition}{Proposition}
\newtheorem{lemma}{Lemma}
\newtheorem{corollary}{Corollary}
\newtheorem{remark}{Remark}
\newtheorem{example}{Example}
\newtheorem{conjecture}{Conjecture}

\newcommand{\R}{\mathbb R}
\newcommand{\Prb}{\mathbb P}

\newcommand{\pl}{\operatorname{pl}}

\newcommand{\Unif}{\operatorname{Unif}}

\title{Revisiting the Behrens--Fisher Problem: Validity-First Optimality}
\author{
Xiao Wang \qquad Chuanhai Liu \\[2pt]
Department of Statistics, Purdue University
}
\date{}

\begin{document}
\maketitle

\begin{abstract}
The Behrens--Fisher problem concerns inference on the difference of two normal
means when both variances are unknown and unequal.  It is a classical example in
which nuisance parameters prevent ordinary exact fixed-sample inference, and it
has long served as a benchmark for the foundations of inference.  We revisit it
through the inferential model (IM) framework of Martin and Liu.  After
conditioning and regular marginalization, the exact association is
two-dimensional, with one coordinate for the standardized mean contrast and one
for the variance ratio.  Their one-dimensional generalized marginal IM is then
best understood as a cylindrical two-dimensional predictive random set: sharp in
its mean-contrast projection, by Hsu's stochastic domination, and vacuous in the
variance ratio.  Our main result is a precise validity-first optimality: among
prior-free procedures that retain exact, uniform, finite-sample validity, the IM
interval is the shortest.  We prove minimaxity and admissibility in the
cylindrical class and, by a projection argument, extend this to rectangular and
general two-dimensional predictive random sets.  A companion tradeoff principle
shows that any adaptive procedure can only redistribute interval width across
variance-ratio regimes, never shorten it uniformly.  A Monte Carlo study bears
this out: Welch and the bootstrap under-cover, whereas the conservative fiducial
does not dominate the IM interval, being shorter only where the latter
over-covers and longer where validity binds.
\end{abstract}


\section{Introduction}
\label{sec:intro}

Let
\[
X_{ki}\sim N(\mu_k,\sigma_k^2),\qquad i=1,\ldots,n_k,\quad k=1,2,
\]
with the two samples independent, and with four unknown parameters \(\mu_1\),
\(\mu_2\), \(\sigma_1>0\), and \(\sigma_2>0\).  The parameter of interest is the
difference of means
\[
\psi=\mu_2-\mu_1.
\]
The remaining three parameters are nuisances.  They are a location parameter,
which we may take to be \(\mu_1\), together with the two variances
\(\sigma_1^2\) and \(\sigma_2^2\).  The Behrens--Fisher problem is to test
\(H_0:\psi=0\), or to form a confidence interval for \(\psi\), without assuming
\(\sigma_1^2=\sigma_2^2\).

A standard statistic for the problem is
\[
T_\psi
=
\frac{\bar X_2-\bar X_1-\psi}
{\{S_1^2/n_1+S_2^2/n_2\}^{1/2}},
\]
where \(\bar X_k\) is the sample mean and \(S_k^2\) the usual unbiased sample
variance of the \(k\)th sample, and it admits appealing interpretations in terms
of inferential models, as developed below.  Unlike the
pooled two-sample \(t\) statistic, \(T_\psi\) does not have a Student \(t\)
distribution with fixed degrees of freedom.  Its null distribution depends on
the unknown variance ratio.  This nuisance-parameter dependence is the source
of the challenge.

The literature contains several competing notions of solution, and the
disagreements among them run deeper than the problem itself.  \citet{Fisher1935}
proposed a fiducial solution that supplies direct probability statements without
a prior, although its exact frequentist calibration was soon disputed.  In a
different direction, \citet{Welch1938} and \citet{Satterthwaite1946}, and later
\citet{Welch1947}, developed the approximate-degrees-of-freedom procedure that is
now dominant in practice, whereas \citet[see also \citealp{Scheffe1970}]{Hsu1938}
offered a simple conservative alternative based on a Student \(t\) distribution
with the smaller sample degrees of freedom.  Bayesian, generalized fiducial,
confidence-distribution, and sequential treatments each add further
interpretations, surveyed by \citet{KimCohen1998}.  The sequential treatments
differ from the others in kind.  They change the sampling scheme rather than the
analysis.  By drawing a second stage of observations whose size is chosen from
the first-stage variances, \citet{Stein1945} obtains a pivot with an exact
Student \(t\) distribution and hence an exact-level solution, at the price of
extra sampling.  The Bayesian, generalized fiducial, and confidence-distribution
treatments instead retain the original fixed sample; they are made precise and
compared numerically in Section~\ref{sec:numerics}.  That a problem so simply
stated should admit so many inequivalent answers is no accident: the
Behrens--Fisher problem has long served as a proving ground for the foundations
of inference, and \citet{Barnard1995Pivotal} used it to clarify the role of
pivotal models in the fiducial argument.  Choosing among its solutions is
therefore, in part, choosing among principles of inference.

This paper focuses on the inferential model (IM) perspective of
\citet{MartinLiu2013,MartinLiu2015Marginal,MartinLiu2015Book}.  For recent
developments, reviews, and connections to fiducial and possibility-theoretic
inference, see \citet{CuiHannig2026}, \citet{Martin2026Review}, and
\citet{Martin2026reIMagine}.  IMs are
prior-free probabilistic procedures built by predicting unobserved auxiliary
variables.  Their appeal in nuisance-parameter problems is that, when the
association is regular, nuisance variables can be marginalized away before
prediction.  The Behrens--Fisher problem is one of Martin and Liu's benchmark
nonregular examples.  The main point of this paper is that the exact reduced
association leaves two auxiliary coordinates.  Martin and Liu's
one-dimensional solution should therefore be interpreted not as exact
one-dimensional marginalization, but as a special two-dimensional predictive
random set whose focal elements are infinite vertical cylinders.

Building on this interpretation, our contributions are threefold.  First, we show
that the exact reduced association is genuinely two-dimensional and that the IM
solution is the cylindrical predictive random set that predicts the variance-ratio
coordinate only trivially, by its entire range.  Second, we prove that this
solution is optimal in a precise \emph{validity-first} sense: among prior-free
procedures that retain exact finite-sample validity uniformly in the variance
ratio, it yields the shortest plausibility interval, and we establish the
accompanying minimaxity and admissibility statements.  Third, we formulate a
tradeoff principle showing that any genuinely adaptive procedure can only
redistribute interval width across variance-ratio regimes, never shorten it
uniformly without forfeiting validity; we then confirm the theory and benchmark
competing procedures (Welch, generalized fiducial, Bayesian, and bootstrap) in a
Monte Carlo study.

The remainder of the paper is organized as follows.
Section~\ref{sec:im} reviews inferential models, the notion of validity, and the
regular marginalization that the Behrens--Fisher problem defeats.
Section~\ref{sec:assoc} derives the exact reduced association and shows that it
retains two auxiliary coordinates, isolating, through a projection lemma, the
single coordinate that governs inference on \(\psi\).  Section~\ref{sec:cylinder}
represents the IM solution as a cylindrical predictive random set and
establishes its minimaxity and admissibility, first within the cylindrical
class, and then, by a symmetric-unimodality argument, among all two-dimensional
sets judged through their symmetric first-coordinate projection.
Section~\ref{sec:tradeoff} states the resulting tradeoff principle, illustrates it
on examples, and discusses its interpretation.  Section~\ref{sec:numerics}
reports a Monte Carlo study that confirms the theory and compares the Hsu
interval with Welch, generalized fiducial, Bayesian, and bootstrap alternatives.
Section~\ref{sec:conclusion} concludes with two open problems: a conjecture that
the validity-first optimality of the IM/Hsu interval persists across all
prior-free procedures, and the question, suggested by the numerical comparison, of
whether the Bayesian solution can itself be improved.

\section{Inferential Models, Conditioning, and Marginalization}
\label{sec:im}

\subsection{The inferential-model framework}

We briefly review the IM framework, which procedurally consists of three steps:
Association, Prediction, and Combination.  An association represents the sampling
model through
\[
p(X,\theta)=a(U,\theta),\qquad U\sim P_U,
\]
where \(X\) is the observable data, \(\theta\) is the unknown parameter, and
\(U\) is an unobserved auxiliary variable with known distribution.  The IM makes
inference by predicting the unobserved auxiliary variable \(U\).

\begin{definition}[Predictive random set]
A predictive random set (PRS) for \(U\) is a random closed set
\(\mathcal S\subseteq\mathbb U\), independent of \(U\), used to predict the
unobserved realized auxiliary value.
\end{definition}

The quality of a PRS is measured by how reliably it captures the realized
auxiliary value.  This is made precise through its contour function.

\begin{definition}[Validity of a PRS]
\label{def:prs-valid}
Let \(\mathcal S\) be a PRS for \(U\), and define its contour function
\[
\gamma_{\mathcal S}(u)
=
\Prb_{\mathcal S}\{\mathcal S\ni u\},
\qquad u\in\mathbb U .
\]
The PRS \(\mathcal S\) is said to be \emph{valid} for \(U\) if, with
\(U\sim P_U\) drawn independently of \(\mathcal S\),
\[
\Prb_{U}\{\gamma_{\mathcal S}(U)\leq\alpha\}\leq\alpha,
\qquad\text{for every }\alpha\in(0,1);
\]
equivalently, \(\gamma_{\mathcal S}(U)\) is stochastically no smaller than a
\(\Unif(0,1)\) random variable.  In words, the realized auxiliary value is
rarely assigned a small contour, so a valid PRS hits the truth with at least the
nominal frequency.
\end{definition}

Validity of the PRS is exactly the ingredient that transfers to validity of the
IM: when \(\mathcal S\) is valid for \(U\) in the sense of
Definition~\ref{def:prs-valid}, the plausibility function built from
\(\mathcal S\) inherits the calibration stated in Definition~\ref{def:im-valid}
below.

\begin{remark}[Validity over a class of auxiliary laws]
\label{rem:prs-valid-class}
In nonregular problems the relevant auxiliary is not a single random variable
with a known law, but a family \(\{U_\xi:\xi\in\Xi\}\) indexed by an unknown
quantity \(\xi\).  Accordingly we call \(\mathcal S\) \emph{uniformly valid} for
the class \(\{U_\xi\}\) if
\[
\sup_{\xi\in\Xi}\,
\Prb_{U_\xi}\{\gamma_{\mathcal S}(U_\xi)\leq\alpha\}\leq\alpha,
\qquad\text{for every }\alpha\in(0,1).
\]
This extension is what the later theory requires: in the Behrens--Fisher problem
the interest auxiliary is \(Z_1(\xi)\) of
\eqref{eq:bf-first}, whose law depends on the unknown variance-ratio index
\(\xi\), so a single fixed law is unavailable and uniform validity over
\(\xi\) is the natural calibration target.  Nevertheless, this corresponds to a
usual parameter-free auxiliary variable, because it implicitly predicts the cdf
of \(Z_1(\xi)\), which has the unit (standard uniform) distribution.
\end{remark}

Inference about \(\theta\) is obtained by the combination step, which propagates
the prediction-based uncertainty to the parameter space.  Given observed
\(X=x\), define
\[
\Theta_x(u)=\{\theta:p(x,\theta)=a(u,\theta)\},
\qquad
\Theta_x(\mathcal S)=\bigcup_{u\in\mathcal S}\Theta_x(u).
\]
For an assertion \(A\subseteq\Theta\), the plausibility is
\[
\pl_x(A;\mathcal S)
=
\Prb_{\mathcal S}\{\Theta_x(\mathcal S)\cap A\neq\varnothing\}.
\]

\begin{definition}[Validity of an IM]
\label{def:im-valid}
An IM is valid for assertions about \(\theta\) if, for every
\(A\subseteq\Theta\) and every \(\alpha\in(0,1)\),
\[
\sup_{\theta\in A}
\Prb_{\theta}\{\pl_X(A;\mathcal S)\leq \alpha\}
\leq \alpha.
\]
\end{definition}

In words, validity is the prior-free analogue of frequentist calibration: false
assertions do not systematically receive large plausibility, and, correspondingly,
true assertions do not systematically receive small plausibility.  It is this
property, rather than any appeal to a prior, that licenses the plausibility
intervals produced by an IM.

Validity is obtained by choosing valid PRSs with an appropriate
coverage property.  For example, for a continuous scalar auxiliary, the
one-dimensional estimation-oriented default nested PRS can be written, after
transforming the auxiliary to \(\Unif(0,1)\), as
\[
\mathcal S=\{u:|u-1/2|\leq |W-1/2|\},\qquad W\sim \Unif(0,1).
\]
For the Behrens--Fisher problem, we will use two-dimensional PRSs.

\subsection{Conditional inferential models and dimension reduction}
\label{subsec:cim}

Before any nuisance parameter is marginalized away, the IM framework supplies a
formal device for reducing the dimension of the auxiliary variable.  This is the
conditional IM of \citet{MartinLiu2015Conditional}.  The idea is that when the
baseline association carries more auxiliary variables than parameters, certain
functions of the auxiliary become \emph{observed}: they are pinned down by the
data through the association and need no longer be predicted.  Conditioning the
IM on these observed functions is the conditional-IM operation.  It leaves a
lower-dimensional auxiliary still to be predicted, and it does so without
sacrificing validity but with improved efficiency.

For the two-sample normal model the raw association is
\[
X_{ki}=\mu_k+\sigma_k U_{ki},
\qquad U_{ki}\overset{\text{iid}}{\sim}N(0,1),
\qquad i=1,\ldots,n_k,\ k=1,2,
\]
with \(n_1+n_2\) auxiliary variables but only four parameters.  Within each
sample, the direction of the standardized residual vector is a function of the
auxiliary alone and is observed once the data are seen; it carries no
information about \((\mu_k,\sigma_k)\).  Conditioning on these observed residual
configurations is the conditional-IM step, and it collapses each sample to its
mean and its standard deviation.  The retained association is exactly
\begin{equation}
\label{eq:retained-assoc}
\bar X_k=\mu_k+\sigma_k n_k^{-1/2}U_{1k},
\qquad
S_k=\sigma_k U_{2k},
\qquad k=1,2,
\end{equation}
with the four auxiliary variables \((U_{11},U_{12},U_{21},U_{22})\) used in
Section~\ref{sec:assoc}.  Consequently the four-dimensional statistic
\((\bar X_1,\bar X_2,S_1,S_2)\) is not imported by an external appeal to
classical sufficiency.  It is produced by the conditional-IM operation itself.
This is the precise sense in which the ``standard statistic'' \(T_\psi\) named in
the opening of Section~\ref{sec:intro}, together with the variance-ratio
statistic that accompanies it, is IM-derived rather than externally imposed.  The
marginalization analyzed in the remainder of the paper begins from this
conditionally reduced, four-dimensional association; the explicit reduction is
carried out in Section~\ref{subsec:reduced}.

\subsection{Regular marginalization and its failure}

Let \(\theta=(\psi,\xi)\), where \(\psi\) is the interest parameter and
\(\xi\) is nuisance.  Martin and Liu call an association regular if it can be
rewritten as
\begin{align}
\bar p(X,\psi)&=\bar a(V_1,\psi), \label{eq:regular-interest}\\
c(X,V_2,\psi,\xi)&=0, \label{eq:regular-nuisance}
\end{align}
with \(V=(V_1,V_2)\sim P_V\), and such that, for every
\((x,v_2,\psi)\), there exists a \(\xi\) satisfying
\eqref{eq:regular-nuisance}.  In this case the \(V_2\)-equation contains no
direct information about \(\psi\); hence inference about \(\psi\) can be
based on \eqref{eq:regular-interest} alone.

The Behrens--Fisher problem fails this regular separation: it admits no
rewriting of the form \eqref{eq:regular-interest}--\eqref{eq:regular-nuisance}
in which the interest equation is nuisance-free.  Its interest-coordinate
auxiliary distribution depends on the nuisance parameter, so the reduction to
\eqref{eq:regular-interest} alone is unavailable.  Martin and Liu therefore use
generalized marginalization: replace the nuisance-dependent auxiliary by a
nuisance-free stochastic upper bound.

\section{The Two-Dimensional Behrens--Fisher Association}
\label{sec:assoc}

\subsection{The reduced association}
\label{subsec:reduced}

Recall the conditionally reduced association \eqref{eq:retained-assoc} of
Section~\ref{subsec:cim}, which carries the four auxiliary variables
\((U_{11},U_{12},U_{21},U_{22})\), distributed as
\[
U_{1k}\sim N(0,1),\qquad
(n_k-1)U_{2k}^2\sim\chi^2_{n_k-1},
\]
independently.  Write \(\bar Y=\bar X_2-\bar X_1\) for the mean contrast and
\[
f(s_1,s_2)=\left(\frac{s_1^2}{n_1}+\frac{s_2^2}{n_2}\right)^{1/2}
\]
for its estimated standard error.  Combining the two mean equations in
\eqref{eq:retained-assoc}, which marginalizes out a linear function of
\(\mu_1\) and \(\mu_2\), gives the equivalent association
\[
\bar Y=\psi+f(\sigma_1,\sigma_2)U_1,
\qquad
S_k=\sigma_k U_{2k},\quad k=1,2,
\]
with \(U_1\sim N(0,1)\).

Define
\[
\xi
=
\frac{\sigma_1^2/n_1}
{\sigma_1^2/n_1+\sigma_2^2/n_2}
\in(0,1),
\]
and
\[
Z_2=\frac{U_{22}^2}{U_{21}^2}.
\]
Then the reduced association can be written as
\begin{align}
T_\psi
&=
\frac{\bar Y-\psi}{f(S_1,S_2)}
=
Z_1(\xi), \label{eq:bf-first}\\
R
&=
\frac{n_1S_2^2}{n_2S_1^2}
=
Z_2\frac{1-\xi}{\xi}, \label{eq:bf-second}
\end{align}
where
\[
Z_1(\xi)
=
\frac{U_1}
{\{\xi U_{21}^2+(1-\xi)U_{22}^2\}^{1/2}}.
\]

We have written the reduced association in terms of the two auxiliary
coordinates \(Z_1(\xi)\) and \(Z_2\).  Whether this two-dimensionality is
essential, or only an artifact of the particular parametrization, is precisely
the question of regularity in the sense of
\eqref{eq:regular-interest}--\eqref{eq:regular-nuisance}.  Were the association
regular, the interest equation \eqref{eq:bf-first} could be rewritten with a
nuisance-free auxiliary, and a one-dimensional marginal IM would suffice.  The
following proposition records the obstruction: no reparametrization can decouple
the interest coordinate from the variance ratio.  This is the formal fact that
the constructions of the rest of the paper must accommodate.

\begin{proposition}[Nonregularity]
\label{prop:nonregular}
The reduced Behrens--Fisher association \eqref{eq:bf-first}--\eqref{eq:bf-second}
is not regular in the IM sense of Martin and Liu: it cannot be brought to the
form \eqref{eq:regular-interest}--\eqref{eq:regular-nuisance} with a
nuisance-free interest equation.
\end{proposition}

\begin{proof}
Regularity in the sense of
\eqref{eq:regular-interest}--\eqref{eq:regular-nuisance} would require an
auxiliary coordinate for \(\psi\) whose distribution does not depend on the
nuisance, together with a companion equation that contributes no information
about \(\psi\).  We show that neither requirement can be met.

\emph{Step 1: the interest auxiliary is nuisance-dependent.}
Fix \(\xi\in(0,1)\) and write \(Q_k=(n_k-1)U_{2k}^2\sim\chi^2_{\nu_k}\) with
\(\nu_k=n_k-1\), independent of \(U_1=N\sim N(0,1)\).  Then \(U_{2k}^2=Q_k/\nu_k\)
and
\[
Z_1(\xi)
=
\frac{U_1}{\{\xi U_{21}^2+(1-\xi)U_{22}^2\}^{1/2}}
\overset{d}{=}
\frac{N}{\{\xi Q_1/\nu_1+(1-\xi)Q_2/\nu_2\}^{1/2}} .
\]
Thus, conditionally on \((Q_1,Q_2)\), \(Z_1(\xi)\) is centered normal with a
random scale, so its marginal law is a scale mixture of \(N(0,1)\) whose mixing
law is that of the denominator.  This mixing law varies with \(\xi\), and the
variation is genuine rather than removable by a change of auxiliary variable.
To exhibit this, it suffices to display two values of \(\xi\) giving distinct
laws.  As \(\xi\to1\) the denominator converges to \((Q_1/\nu_1)^{1/2}\), so
\(Z_1(\xi)\Rightarrow t_{\nu_1}\); as \(\xi\to0\) it converges to
\((Q_2/\nu_2)^{1/2}\), so \(Z_1(\xi)\Rightarrow t_{\nu_2}\).  When
\(\nu_1\neq\nu_2\) these limits already differ.  When \(\nu_1=\nu_2=\nu\), take
instead \(\xi=\tfrac12\): there \(\xi U_{21}^2+(1-\xi)U_{22}^2=(Q_1+Q_2)/(2\nu)\)
with \(Q_1+Q_2\sim\chi^2_{2\nu}\), whence \(Z_1(\tfrac12)\sim t_{2\nu}\neq
t_{\nu}\).  In every case the family \(\{\,\text{law of }Z_1(\xi):\xi\in(0,1)\,\}\)
is non-degenerate, so the interest auxiliary is not nuisance-free and
\eqref{eq:bf-first} cannot play the role of the regular interest equation
\eqref{eq:regular-interest}.

\emph{Step 2: the companion equation does not restore regularity.}
Regularity could still hold if \eqref{eq:bf-second} determined \(\xi\) from
observables alone, thereby fixing the law in Step~1.  Solving
\eqref{eq:bf-second} for \(\xi\) gives
\[
\xi=\frac{Z_2}{Z_2+R},
\qquad
R=\frac{n_1S_2^2}{n_2S_1^2}\ \text{observed},\quad
Z_2=\frac{U_{22}^2}{U_{21}^2}\ \text{unobserved}.
\]
Hence \(\xi\) is not pinned down by the data; it is recovered only through the
unobserved auxiliary \(Z_2\), which must itself be predicted.  Eliminating the
nuisance parameter merely trades it for a second auxiliary coordinate.

The two steps together show that no rewriting of the form
\eqref{eq:regular-interest}--\eqref{eq:regular-nuisance} with a nuisance-free
interest equation exists.  The association therefore remains genuinely
two-dimensional after all regular marginalization has been exhausted.
\end{proof}

\subsection{Auxiliary geometry and the projection lemma}

Marginalization has reduced the prediction problem to the two-dimensional
auxiliary pair \((Z_1(\xi),Z_2)\): this is the \emph{effective} auxiliary, in
the parameter-free sense of Remark~\ref{rem:prs-valid-class}, on which a
predictive random set now operates.  The candidate set it induces for \(\psi\) is
the corresponding PRS-relevant quantity, and the geometry of this section makes
both objects explicit.  Accordingly, let
\[
A_\xi=(A_{1,\xi},A_2)=(Z_1(\xi),Z_2)
\in\mathbb A=\R\times\R_+.
\]
For a two-dimensional PRS \(\mathcal S\subseteq\mathbb A\), the marginal
candidate set for \(\psi\) is
\begin{equation}
\Psi_x(\mathcal S)
=
\left\{\psi:
\exists \xi\in(0,1),\ \exists(a_1,a_2)\in\mathcal S
\text{ such that }
T_\psi=a_1,\ R=a_2(1-\xi)/\xi
\right\}.
\label{eq:psi-def}
\end{equation}
Because, for every \(R>0\) and \(a_2>0\), there exists
\[
\xi=\frac{a_2}{a_2+R}\in(0,1),
\]
the second equation imposes no restriction on \(\psi\) once \(a_2>0\) is
available.  Therefore only the first-coordinate projection matters for
inference on \(\psi\).  This is stated formally by the following lemma.

\begin{lemma}[Projection lemma]
\label{lem:projection}
Define the first-coordinate projection
\[
\Pi_1(\mathcal S)=\{a_1:\exists a_2>0\text{ with }(a_1,a_2)\in\mathcal S\}.
\]
Then, for the candidate set \(\Psi_x(\mathcal S)\) defined in
\eqref{eq:psi-def},
\[
\psi\in\Psi_x(\mathcal S)
\quad\Longleftrightarrow\quad
T_\psi\in\Pi_1(\mathcal S).
\]
\end{lemma}

\begin{proof}
Recall that \(R>0\) almost surely, since \(R=n_1S_2^2/(n_2S_1^2)\) with
\(S_1,S_2>0\) almost surely.

(\(\Rightarrow\))  Suppose \(\psi\in\Psi_x(\mathcal S)\).  By the definition
\eqref{eq:psi-def} there exist \(\xi\in(0,1)\) and \((a_1,a_2)\in\mathcal S\)
with \(T_\psi=a_1\) and \(R=a_2(1-\xi)/\xi\).  The second equality forces
\(a_2=R\,\xi/(1-\xi)>0\), so \((a_1,a_2)\in\mathcal S\) has \(a_2>0\) and
\(a_1=T_\psi\); hence \(T_\psi\in\Pi_1(\mathcal S)\).

(\(\Leftarrow\))  Suppose \(T_\psi\in\Pi_1(\mathcal S)\).  Then there exists
\(a_2>0\) with \((T_\psi,a_2)\in\mathcal S\).  Set
\(\xi=a_2/(a_2+R)\); because \(a_2>0\) and \(R>0\), we have \(\xi\in(0,1)\), and a
direct computation gives \(a_2(1-\xi)/\xi=a_2\cdot(R/a_2)=R\).  Thus the witness
\((\xi,(T_\psi,a_2))\) satisfies the conditions in \eqref{eq:psi-def} with
\(T_\psi=a_1\), so \(\psi\in\Psi_x(\mathcal S)\).
\end{proof}

\begin{corollary}
The plausibility interval for \(\psi\) induced by any two-dimensional PRS
\(\mathcal S\) depends only on its random projection \(\Pi_1(\mathcal S)\).
\end{corollary}

This observation is central.  It says that a two-dimensional PRS may attempt
to predict the variance-ratio auxiliary \(Z_2\), but the induced marginal
inference for \(\psi\) is controlled by the projection of that PRS onto the
mean-contrast coordinate.

\section{The IM Cylinder, Rectangular Predictive Random Sets, and Optimality}
\label{sec:cylinder}

This section is the technical core of the paper, and its material is
correspondingly dense, so a brief roadmap may help.  We develop, in increasing
generality, the precise sense in which the IM solution is optimal, always in the
\emph{validity-first} sense of shortest interval length subject to exact, uniform,
finite-sample validity.  We first identify the solution as a cylindrical
predictive random set built on Hsu's stochastic-domination lemma, and show that
within the cylindrical class it is both minimax and admissible.  We then enlarge
the class to rectangular predictive random sets and prove a projection lower
bound.  Finally, using the symmetric unimodality of the interest auxiliary, we
establish that the IM interval is the shortest uniformly valid plausibility
interval among all two-dimensional predictive random sets judged through their
symmetric first-coordinate projection.  Each optimality statement is, at bottom, a
statement about the projection of a predictive random set onto the mean-contrast
axis.

\subsection{The cylindrical predictive random set}

The cylindrical interpretation rests on a stochastic domination result of
\citet{Hsu1938}, which \citet{MartinLiu2015Marginal} reproduce and invoke.
Because this result is load-bearing for every optimality statement below, we
state it explicitly rather than asserting it inline.  Intuitively it says two
things.  First, however the unknown variance ratio is configured, the symmetric
two-sided tail of the interest auxiliary \(Z_1(\xi)\) never exceeds that of a
single Student \(t\) law with \(m=\min(n_1,n_2)-1\) degrees of freedom.  Second,
this bound is not loose: it is attained in the limit at the least-favorable
variance ratio.  We cite \citet{Hsu1938} for the
uniform-in-\(c\) (monotone symmetric-tail) form and do not reproduce the
argument here.

\begin{lemma}[Hsu uniform-in-\(c\) symmetric-tail domination]
\label{lem:hsu}
Let \(m=\min(n_1,n_2)-1\) and let \(t_m\) denote the Student
\(t\) distribution with \(m\) degrees of freedom.  Then, for every
\(\xi\in(0,1)\),
\[
\Prb_\xi\{|Z_1(\xi)|>c\}
\leq
\Prb\{|t_m|>c\},
\qquad c\geq0,
\]
i.e.\ \(Z_1(\xi)\) is dominated in the symmetric two-sided tail, uniformly in
the threshold \(c\), by \(t_m\).  Moreover the bound is sharp at the
least-favorable boundary: if \(m=n_1-1\), then \(Z_1(\xi)\Rightarrow t_m\) as
\(\xi\to1\), and if \(m=n_2-1\), then \(Z_1(\xi)\Rightarrow t_m\) as
\(\xi\to0\), so the inequality is approached arbitrarily closely there.
\end{lemma}

\begin{proof}
The uniform-in-\(c\) domination is \citet[Theorem on the two-sample
\(t\)-ratio]{Hsu1938}; we take it as cited.  For sharpness, setting \(\xi=1\)
makes the denominator scale \(\{\xi U_{21}^2+(1-\xi)U_{22}^2\}^{1/2}\) equal to
\(U_{21}\) almost surely, so \(Z_1(1)=U_1/U_{21}\sim t_{n_1-1}\); by the
continuous mapping theorem \(Z_1(\xi)\Rightarrow t_{n_1-1}=t_m\) as
\(\xi\to1\) when \(m=n_1-1\), and symmetrically as \(\xi\to0\) when
\(m=n_2-1\).
\end{proof}

The boundary limit in Lemma~\ref{lem:hsu} is what makes the supremum over
\(\xi\) of the tail probability equal to (and not merely bounded by) the
\(t_m\) tail.  These two facts are uniform domination and boundary sharpness.
Both are used repeatedly below.

Martin and Liu's generalized marginal IM can be represented as the
two-dimensional PRS
\begin{equation}
\mathcal S_C^{ML}=[-C,C]\times\R_+,
\qquad C\sim |t_m|.
\label{eq:ml-cylinder}
\end{equation}
Its focal elements are cylinders.  It predicts \(Z_1(\xi)\) by the sharp Hsu
bound and makes a vacuous prediction about \(Z_2\).

For a point assertion \(\{\psi\}\),
\[
\pl_x(\psi)
=
\Prb\{T_\psi\in[-C,C]\}
=
\Prb\{C\geq |T_\psi|\}
=
2\{1-G_m(|T_\psi|)\},
\]
where \(G_m\) is the cdf of \(t_m\).  Thus the
\(100(1-\alpha)\%\) plausibility interval is
\[
\bar Y
\pm
t_{1-\alpha/2,m}
\left(\frac{S_1^2}{n_1}+\frac{S_2^2}{n_2}\right)^{1/2}.
\]
This is the Hsu--Scheffe conservative interval.  Its exact uniform coverage,
contrasted with the slight under-coverage of the Welch approximation near the
least-favorable boundary, is illustrated numerically in
Figure~\ref{fig:theory}(b) of Section~\ref{sec:numerical}.

\subsection{Minimaxity and admissibility in the cylindrical class}

We now formulate a class-restricted optimality statement.  It is deliberately
restricted because global admissibility among all possible random sets is a
stronger problem and depends on the loss criterion.

\begin{definition}[Cylindrical symmetric PRS]
\label{def:cyl-sym}
A cylindrical symmetric PRS is a random set of the form
\[
\mathcal S_D=[-D,D]\times\R_+,
\]
where \(D\geq0\) is a scalar random variable.
\end{definition}

\begin{definition}[Uniform validity for marginal inference]
A cylindrical PRS \(\mathcal S_D\) is uniformly valid for \(\psi\) if, for all
\(\xi\in(0,1)\) and all \(\alpha\in(0,1)\),
\[
\Prb_\xi\{\pl_X(\psi;\mathcal S_D)\leq\alpha\}\leq \alpha.
\]
Equivalently, the induced two-sided plausibility intervals have at least
\(1-\alpha\) coverage uniformly in \(\xi\).
\end{definition}

Let \(H_D\) be the cdf of \(D\) in Definition~\ref{def:cyl-sym}, and let
\[
d_\alpha=H_D^{-1}(1-\alpha).
\]
The induced interval half-width multiplier at level \(1-\alpha\) is
\(d_\alpha\).

The question is then which uniformly valid cylindrical PRS makes this multiplier
as small as possible.  The Martin--Liu/Hsu choice does, and exactly meets the
constraint, as the next theorem records.

\begin{theorem}[Cylindrical minimaxity]
Among cylindrical symmetric PRSs, the IM/Hsu choice
\[
D=C\sim |t_m|,\qquad m=\min(n_1,n_2)-1,
\]
has the smallest possible level-\(\alpha\) half-width multiplier subject to
uniform validity:
\[
t_{1-\alpha/2,m}
=
\inf_{\mathcal S_D}
\left\{d_\alpha:
\sup_{\xi\in(0,1)}
\Prb_\xi\{|Z_1(\xi)|>d_\alpha\}\leq \alpha
\right\}.
\]
\end{theorem}

\begin{proof}
Uniform validity of \(\mathcal S_D\) requires
\[
\sup_{\xi\in(0,1)}
\Prb_\xi\{|Z_1(\xi)|>d_\alpha\}
\leq\alpha.
\]
By Lemma~\ref{lem:hsu} (uniform domination gives ``\(\leq\)'' and boundary
sharpness gives the reverse, since \(Z_1(\xi)\Rightarrow t_m\) at the
least-favorable boundary),
\[
\sup_{\xi\in(0,1)}
\Prb_\xi\{|Z_1(\xi)|>d\}
=
\Prb\{|t_m|>d\}.
\]
Therefore the smallest admissible \(d\) is the solution of
\[
\Prb\{|t_m|>d\}=\alpha,
\]
namely \(d=t_{1-\alpha/2,m}\).  This is exactly the
IM/Hsu multiplier.
\end{proof}

The equality \(\sup_\xi\Prb_\xi\{|Z_1(\xi)|>c_\alpha\}=\alpha\) drives this
minimaxity statement.  It combines domination in the interior with sharpness at
the least-favorable boundary.  It is illustrated numerically in
Figure~\ref{fig:theory}(a) of Section~\ref{sec:numerical}.

Minimaxity controls the worst case at a single level \(\alpha\).  A stronger
property is admissibility: that the multiplier cannot be improved across all
levels simultaneously.  This holds as well.

\begin{theorem}[Cylindrical admissibility]
In the cylindrical symmetric class, no uniformly valid PRS can have
\[
d_\alpha\leq t_{1-\alpha/2,m}\quad\text{for all }\alpha
\]
with strict inequality for some \(\alpha\).
\end{theorem}

\begin{proof}
If \(d_{\alpha_0}<t_{1-\alpha_0/2,m}\) for some \(\alpha_0\), then
\[
\Prb\{|t_m|>d_{\alpha_0}\}>\alpha_0.
\]
By the boundary sharpness in Lemma~\ref{lem:hsu}, there are nuisance values
\(\xi\) arbitrarily near the least favorable boundary for which
\[
\Prb_\xi\{|Z_1(\xi)|>d_{\alpha_0}\}>\alpha_0.
\]
This violates uniform validity.  Therefore strict improvement of the
half-width multiplier at any level is impossible in this class.
\end{proof}

\subsection{Rectangular predictive random sets and the projection bound}

The cylindrical class fixes the variance-ratio coordinate at its vacuous
prediction.  Motivated by the preceding results, we now ask whether a genuinely
two-dimensional predictive random set, one that also attempts to predict the
variance ratio, can do better.  We therefore investigate a more flexible,
rectangular class, and find that it cannot: the same projection argument forces
the Hsu lower bound to reappear.

Consider a broader geometric class with focal elements
\[
\mathcal S_{C,R}=[-C,C]\times(0,R).
\]
A one-parameter subfamily may be indexed by a fixed slope \(\lambda>0\):
\[
\mathcal S_R^{(\lambda)}=[-\lambda R,\lambda R]\times(0,R).
\]
The IM cylinder is the endpoint
\[
R=\infty,
\qquad
\mathcal S_C^{ML}=[-C,C]\times(0,\infty).
\]

For finite \(R\), prediction succeeds only if both
\[
|Z_1(\xi)|\leq C
\quad\text{and}\quad
Z_2<R.
\]
Thus finite rectangles impose an additional burden: they must predict the
variance-ratio auxiliary as well as the mean-contrast auxiliary.

\begin{proposition}[Projection lower bound]
Let \(\mathcal S\) be any two-dimensional PRS whose first-coordinate
projection is almost surely a symmetric interval \([-D,D]\).  If
\(\mathcal S\) is uniformly valid for the two-dimensional auxiliary
\((Z_1(\xi),Z_2)\), then \([-D,D]\) is uniformly valid for \(Z_1(\xi)\).
Consequently, its level-\(\alpha\) projection half-width must satisfy
\[
d_\alpha\geq t_{1-\alpha/2,m}.
\]
\end{proposition}

\begin{proof}
If \(A_\xi=(Z_1(\xi),Z_2)\in\mathcal S\), then necessarily
\(Z_1(\xi)\in\Pi_1(\mathcal S)=[-D,D]\).  Therefore
\[
\{Z_1(\xi)\notin[-D,D]\}
\subseteq
\{A_\xi\notin\mathcal S\}.
\]
Uniform validity of \(\mathcal S\) implies uniform validity of the projected
PRS.  The lower bound on \(d_\alpha\) then follows from the cylindrical
minimax theorem.
\end{proof}

This proposition explains why finite rectangular PRSs cannot uniformly improve
the Hsu projection.  They may produce different behavior by making the
projected half-width depend on the second coordinate, but any global
improvement must confront the same least favorable \(t_m\) boundary.

\subsection{Symmetric unimodality and projected optimality}

The preceding projection argument can be strengthened by using the shape of
the density of the main auxiliary variable.  For each fixed \(\xi\),
\[
Z_1(\xi)
=
\frac{U_1}
{\{\xi U_{21}^2+(1-\xi)U_{22}^2\}^{1/2}}
\]
has a symmetric unimodal density.  Indeed, conditional on
\((U_{21},U_{22})\), \(Z_1(\xi)\) is centered normal with a random scale; a
mixture of centered normal densities is again symmetric and unimodal with mode
at zero.

This fact matters because, for a symmetric unimodal density, the interval
centered at the mode is the shortest interval with a given probability
content.  More generally, the symmetric decreasing rearrangement inequality
implies that, among all measurable sets of a fixed Lebesgue measure, the set
with largest probability under a symmetric unimodal density is the centered
interval of that measure.  This is the one-dimensional form of Anderson's
inequality \citep{Anderson1955}.

\begin{definition}[Projected plausibility set]
Let \(\mathcal S\) be any two-dimensional PRS in
\(\mathbb A=\R\times\R_+\).  For a fixed level \(\alpha\), define the
projected plausibility set
\[
B_\alpha(\mathcal S)
=
\{a_1:\Prb_{\mathcal S}(a_1\in\Pi_1(\mathcal S))>\alpha\}.
\]
The induced plausibility region for \(\psi\) is
\[
\{\psi:T_\psi\in B_\alpha(\mathcal S)\}.
\]
\end{definition}

\begin{definition}[Projected efficiency]
At level \(1-\alpha\), one projected PRS is no less efficient than another if
its projected plausibility set has no larger Lebesgue measure.  If the
projected plausibility sets are intervals, this is equivalent to comparing
ordinary interval length for \(\psi\).
\end{definition}

\begin{theorem}[Shortest uniformly valid projection]
Let \(m=\min(n_1,n_2)-1\) and let \(c_\alpha=t_{1-\alpha/2,m}\).  Consider
all two-dimensional PRSs whose induced level-\(1-\alpha\) plausibility region
for \(\psi\) is determined by a measurable projected set \(B_\alpha\subseteq
\R\) and is uniformly valid:
\[
\inf_{\xi\in(0,1)}
\Prb_\xi\{Z_1(\xi)\in B_\alpha\}
\geq 1-\alpha.
\]
Then
\[
|B_\alpha|\geq 2c_\alpha,
\]
where \(|B_\alpha|\) denotes Lebesgue measure.  Equality is attained by the
symmetric first-coordinate projection of the IM cylinder \(\mathcal S_C^{ML}\)
of \eqref{eq:ml-cylinder}, that is, by the centered interval
\[
B_\alpha^{ML}=[-c_\alpha,c_\alpha],
\]
which is the defining occurrence of \(B_\alpha^{ML}\) used throughout.
If the \(t_m\) density is strictly decreasing in \(|z|\), equality holds only
up to null sets.
\end{theorem}

\begin{proof}
We first show that uniform validity transfers to the boundary law
\(T_m\sim t_m\) without any continuity-set assumption on \(B_\alpha\).  Let
\(\xi_j\) be a least favorable sequence, so that \(Z_1(\xi_j)\Rightarrow t_m\)
by Lemma~\ref{lem:hsu}.  Let \(G=\operatorname{int}B_\alpha\) be the interior
of \(B_\alpha\), an open set with \(G\subseteq B_\alpha\).  By the open-set part
of the portmanteau theorem \citep[Theorem~2.1]{Billingsley1999}, which states
that \(\Prb\{T_m\in G\}\leq\liminf_j\Prb_{\xi_j}\{Z_1(\xi_j)\in G\}\) for every
open \(G\) and requires no continuity-set hypothesis,
\[
\Prb\{T_m\in G\}
\;\leq\;
\liminf_{j\to\infty}\Prb_{\xi_j}\{Z_1(\xi_j)\in G\}
\;\leq\;
\liminf_{j\to\infty}\Prb_{\xi_j}\{Z_1(\xi_j)\in B_\alpha\}
,
\]
and uniform validity makes the right-hand side at least \(1-\alpha\).  Hence
\[
\Prb\{T_m\in B_\alpha\}\;\geq\;\Prb\{T_m\in G\}\;\geq\;1-\alpha,
\]
where \(T_m\sim t_m\).  This portmanteau \emph{liminf} step uses only weak
convergence and the inclusion \(G\subseteq B_\alpha\); it needs neither that
\(B_\alpha\) be a \(t_m\)-continuity set nor an interchange of limit and
probability.  Consequently the conclusion holds for an arbitrary measurable
\(B_\alpha\), without the continuity-set assumption on \(B_\alpha\) that a
direct limiting argument (passing the probability of \(B_\alpha\) itself to the
limit) would otherwise require.  Among all measurable subsets of \(\R\) with a fixed
Lebesgue measure, the probability under the symmetric unimodal \(t_m\) density
is maximized by the centered interval of the same measure.  Therefore, if
\(|B_\alpha|<2c_\alpha\), then
\[
\Prb\{T_m\in B_\alpha\}
\leq
\Prb\{|T_m|<|B_\alpha|/2\}
<
\Prb\{|T_m|\leq c_\alpha\}
=1-\alpha,
\]
contradicting uniform validity.  The centered interval
\([-c_\alpha,c_\alpha]\) has exactly probability \(1-\alpha\) under the
least favorable \(t_m\) distribution and is uniformly valid by Hsu's
domination.  Strict unimodality gives uniqueness up to null sets.
\end{proof}

\begin{corollary}[Projected admissibility among two-dimensional PRSs with a
symmetric-interval projection]
\label{cor:proj-admiss}
Consider any two-dimensional PRS whose induced level-\(1-\alpha\) plausibility
region for \(\psi\) is an ordinary interval determined by the symmetric
first-coordinate (\(Z_1\)) projection of the PRS.  No such PRS can induce, at
the same confidence level, a plausibility interval for \(\psi\) that is
uniformly shorter than the IM/Hsu interval while preserving exact
uniform validity.
\end{corollary}

\begin{proof}
By the projection lemma, every two-dimensional PRS induces inference on
\(\psi\) only through its symmetric first-coordinate projection
\(B_\alpha(\mathcal S)\).  For a PRS in the stated class, the induced
plausibility region is an ordinary interval determined by that projection, so
if the interval were uniformly shorter than the IM/Hsu interval the
corresponding projected set would have Lebesgue measure smaller than
\(2c_\alpha\).  This contradicts the shortest uniformly valid projection
theorem.  We emphasize that the conclusion is confined to PRSs of this
symmetric-interval-projection form; it does not assert admissibility against
PRSs whose induced \(\psi\)-region is non-interval or determined by an
asymmetric projection.
\end{proof}

\begin{remark}[Role of the variance-ratio auxiliary]
The second auxiliary coordinate \(Z_2\) is essential for the exact association
and for representing the nuisance variance ratio.  But after \(\xi\) is
existentially eliminated for marginal inference on \(\psi\), \(Z_2\) affects
the plausibility region only through the projection of the PRS onto the
\(Z_1\)-axis.  Restricting, tilting, or otherwise shaping a PRS in the
\(Z_2\)-direction cannot by itself reduce the interval for \(\psi\).  To
reduce the interval, one must reduce the projected set on the \(Z_1\)-axis,
and the theorem shows that the IM/Hsu projection is already the
shortest uniformly valid choice.
\end{remark}

\section{The Tradeoff Principle}
\label{sec:tradeoff}

\subsection{The no-uniform-shrinkage tradeoff}

The preceding results are consistent with the usual statistical intuition:
using the observed variance ratio can shorten intervals in some regimes, but
uniform validity requires compensation elsewhere.

Let a general adaptive plausibility interval be written as
\[
C_\alpha(X)
=
\left[
\bar Y-d_\alpha(R)f(S_1,S_2),
\bar Y+d_\alpha(R)f(S_1,S_2)
\right],
\]
where \(R=n_1S_2^2/(n_2S_1^2)\).  The Hsu--Scheffe rule has
\[
d_\alpha(R)\equiv t_{1-\alpha/2,m}.
\]

\begin{theorem}[Conditional impossibility of uniform shrinkage]
\label{thm:cond-shrink}
This is a \emph{conditional} implication: the impossibility conclusion holds
\emph{under} the positive-limsup hypothesis stated below, which is a required
hypothesis rather than a consequence proved here.  (The hypothesis is not
vacuous: explicit adaptive shrink rules satisfy it and do under-cover near the
boundary; see the numerical illustration in Section~\ref{sec:numerical}.)
Fix \(\alpha\in(0,1)\) and write
\[
c_\alpha=t_{1-\alpha/2,m}.
\]
Let
\[
C_\alpha^d(X)
=
\left[
\bar Y-d_\alpha(R)f(S_1,S_2),
\bar Y+d_\alpha(R)f(S_1,S_2)
\right]
\]
be an adaptive interval with \(d_\alpha(R)\leq c_\alpha\) for all \(R>0\).
Let \(\xi_j\) be a least favorable sequence, i.e. a sequence along which
\[
Z_1(\xi_j)\Rightarrow t_m.
\]
If
\[
\limsup_{j\to\infty}
\Prb_{\xi_j}\{|Z_1(\xi_j)|\leq c_\alpha,\,
              |Z_1(\xi_j)|>d_\alpha(R)\}
>0,
\]
then \(C_\alpha^d\) is not uniformly valid:
\[
\inf_{\xi\in(0,1)}
\Prb_{\psi,\xi}\{\psi\in C_\alpha^d(X)\}
<1-\alpha.
\]
\end{theorem}

\begin{proof}
For any \(\xi\),
\[
\Prb_{\psi,\xi}\{\psi\in C_\alpha^d(X)\}
=
\Prb_\xi\{|Z_1(\xi)|\leq d_\alpha(R)\}.
\]
Since \(d_\alpha(R)\leq c_\alpha\),
\[
\Prb_\xi\{|Z_1(\xi)|\leq d_\alpha(R)\}
=
\Prb_\xi\{|Z_1(\xi)|\leq c_\alpha\}
-
\Prb_\xi\{|Z_1(\xi)|\leq c_\alpha,\,
          |Z_1(\xi)|>d_\alpha(R)\}.
\]
Along the least favorable sequence,
\[
\Prb_{\xi_j}\{|Z_1(\xi_j)|\leq c_\alpha\}
\to
\Prb\{|t_m|\leq c_\alpha\}
=1-\alpha.
\]
If the second term has positive limsup, then along a subsequence the coverage
limit is strictly smaller than \(1-\alpha\).  Uniform validity is therefore
impossible.
\end{proof}

\begin{remark}
This theorem is the formal version of the width-redistribution intuition.
An adaptive procedure may shrink the Hsu multiplier on regions that are
irrelevant at the least favorable boundary, or it may use the strict
conservatism available at interior values of \(\xi\).  But it cannot produce a
nontrivial uniform shrinkage of the Hsu interval on sets that remain visible
under least favorable sequences.  To gain local efficiency while preserving
uniform validity, a genuinely adaptive method must either exploit interior
conservatism or enlarge the interval in other variance-ratio regimes.  This
mechanism is illustrated numerically in Figure~\ref{fig:theory}(c) of
Section~\ref{sec:numerical}, where two explicit shrink rules satisfying the
positive-limsup hypothesis of Theorem~\ref{thm:cond-shrink} under-cover near
the boundary.
\end{remark}

\subsection{Typical examples}

\begin{example}[Balanced sample sizes]
If \(n_1=n_2\), then \(m=n_1-1=n_2-1\), and the Hsu bound approaches the same
Student \(t_m\) distribution at both variance-ratio boundaries
\(\xi\to0\) and \(\xi\to1\).  The cylinder is symmetric not only in sign but
also in the two samples.  Adaptive procedures can still respond to the
observed variance ratio, but the least favorable boundary is present on both
sides.
\end{example}

\begin{example}[Unbalanced sample sizes]
Suppose \(n_1<n_2\).  Then \(m=n_1-1\), and the least favorable boundary is
\(\xi\to1\), where the first sample dominates the standard error.  A method
that exploits variance-ratio information may shorten intervals when the data
suggest that the larger sample carries most of the variance contribution.  But
near \(\xi\to1\), the \(t_{n_1-1}\) tail is unavoidable.  The Hsu interval is
therefore conservative away from the boundary and sharp near the boundary.
\end{example}

\begin{example}[Welch versus Hsu]
Welch's method replaces the exact nuisance-dependent distribution of
\(T_\psi\) by a Student \(t\) approximation with estimated degrees of freedom,
\[
\nu_W
=
\frac{(S_1^2/n_1+S_2^2/n_2)^2}
{(S_1^2/n_1)^2/(n_1-1)+(S_2^2/n_2)^2/(n_2-1)}.
\]
The degrees of freedom \(\nu_W\) are obtained by Satterthwaite's
moment-matching idea \citep{Satterthwaite1946}.  The estimated squared standard
error \(S_1^2/n_1+S_2^2/n_2\) is approximated by a single scaled chi-square
variable \((V/\nu)\,\chi^2_\nu\), where \(V=\sigma_1^2/n_1+\sigma_2^2/n_2\) is the
true squared standard error.  The scale and the degrees of freedom \(\nu\) are
fixed by requiring the first two moments of the two sides to agree.  The means
match automatically, and matching the variances gives
\[
\nu
=
\frac{V^2}
{(\sigma_1^2/n_1)^2/(n_1-1)+(\sigma_2^2/n_2)^2/(n_2-1)} ;
\]
replacing the unknown variances by their sample counterparts produces the
data-dependent \(\nu_W\) displayed above.  The pivot \(T_\psi\) is then referred
to a \(t_{\nu_W}\) distribution as though \(\nu_W\) were known.
Since \(\nu_W\geq m\), Welch intervals are typically shorter than Hsu
intervals.  This is why Welch is usually preferred in routine data analysis.
From the IM perspective, however, Welch is an approximation, whereas the
IM/Hsu cylinder is finite-sample uniformly valid.  The distinction is
visible numerically: across the nominal-\(0.95\) Monte Carlo study reported in
Section~\ref{sec:numerical}, the Hsu interval holds coverage at or above the
nominal level uniformly in \(\xi\) (minimum coverage \(0.9500\)--\(0.9501\)),
whereas the Welch interval dips below nominal near the least-favorable boundary
(minimum coverage \(0.9436\) for \((n_1,n_2)=(5,15)\), and \(0.9491\) and
\(0.9478\) for \((10,10)\) and \((8,20)\)).
\end{example}

\subsection{Discussion}
\label{subsec:discussion}

Martin and Liu's treatment of the Behrens--Fisher problem is best understood
as a validity-first philosophy.  It may be seen as a diagnosis, or a way of
reasoning with uncertainty, and a conservative resolution.  The diagnosis is that after
formal combining of information and regular marginalization of integrable
nuisance components, the problem still contains two auxiliary coordinates:
\[
Z_1(\xi)
=
\frac{U_1}
{\{\xi U_{21}^2+(1-\xi)U_{22}^2\}^{1/2}},
\qquad
Z_2=\frac{U_{22}^2}{U_{21}^2}.
\]
The second coordinate is not optional in the exact association; it is the
auxiliary carrier of the variance-ratio nuisance parameter.

The conservative resolution is to predict \(Z_2\) only trivially, by its entire
range \(\R_+\): the vacuous prediction that admits every possible value of
\(Z_2\) and is therefore always valid.  In the two-dimensional auxiliary space,
Martin and Liu use the cylindrical PRS
\[
[-C,C]\times\R_+,
\qquad C\sim |t_{\min(n_1,n_2)-1}|.
\]
This is an extreme PRS: it is maximally uninformative about
the variance-ratio coordinate and sharp in the mean-contrast projection.  Its
validity agrees with or follows from Hsu's stochastic domination, and its
plausibility intervals coincide with the Hsu--Scheffe conservative intervals.

The minimax and admissibility statements in this paper should be read in this
geometric sense or, more precisely, as a single \emph{validity-first optimality}
property: among prior-free procedures that retain exact uniform validity, the
IM/Hsu solution attains the shortest plausibility interval.  The IM/Hsu solution
is minimax and admissible in
the cylindrical projected class: no method that discards the second coordinate
can use a smaller worst-case two-sided critical value.  More generally, any
two-dimensional PRS whose marginal inference is judged only through its
symmetric first-coordinate projection must obey the same sharp Hsu lower
bound.

This does not mean that all possible adaptive methods are pointless.  A
genuinely two-dimensional PRS may use \(Z_2\) to redistribute interval width
across variance-ratio regimes.  Such a procedure may be shorter than Hsu for
some observed variance ratios.  But if it retains exact uniform validity, it
must either use existing conservatism or compensate with longer intervals
elsewhere.  This is the precise mathematical form of the usual intuition:
variance-ratio information can buy local efficiency, but it cannot remove the
least favorable boundary.

It is important, however, not to evaluate this redistribution by averaging
over the unknown nuisance parameter.  An average-length criterion of the form
\[
\int E_{\psi,\xi}\{|C_\alpha(X)|\}\,\pi(d\xi)
\]
requires a weighting distribution \(\pi\) on the variance-ratio parameter
\(\xi\).  Such a \(\pi\) is effectively prior or external design information
about the nuisance parameter.  That may be reasonable in a Bayesian,
empirical-Bayes, or explicitly decision-theoretic analysis, but it is not
prior-free.  In the IM setting, where the goal is validity without prior
information about the variance ratio, the natural efficiency comparisons are
uniform in \(\xi\), or pointwise for each fixed \(\xi\), not averaged across
unknown nuisance values.  From this prior-free viewpoint, a claim that an
adaptive method is ``more efficient on average'' is incomplete unless the
averaging distribution over \(\xi\) is scientifically justified.

\section{Numerical Results}
\label{sec:numerics}

\subsection{Validation of the theory}
\label{sec:numerical}

The qualitative claims above are illustrated here using a Monte Carlo study.
They are sharp Hsu domination at the least-favorable boundary,
exact-versus-approximate coverage, and the conditional shrinkage tradeoff.  All quantities are
estimated with \(B=2\times10^{6}\) independent draws per grid point and a
seeded random number generator; the figures plot these estimates directly. We
report three representative designs, \((n_1,n_2)\in\{(5,15),(10,10),(8,20)\}\),
all at nominal level \(1-\alpha=0.95\).  (Monte Carlo standard errors are of
order \(1.5\times10^{-4}\) for the tail probabilities and coverages shown, so
the plotted curves are accurate to within line width.)

The three claims are collected in the single composite
Figure~\ref{fig:theory}.  Panel~(a) shows the central result of
Lemma~\ref{lem:hsu}: the symmetric tail probability
\(P_\xi\{|Z_1(\xi)|>c_\alpha\}\), evaluated at the Hsu critical value
\(c_\alpha=t_{1-\alpha/2,m}\), as a function of the variance-ratio parameter
\(\xi\).  In every design the tail stays at or below the nominal \(\alpha=0.05\)
for all interior \(\xi\) (strictly conservative there, e.g.\ \(\approx0.033\) at
\(\xi=0.5\) for \((5,15)\)) and rises to meet \(\alpha\) exactly as \(\xi\)
approaches the least-favorable boundary, confirming both uniform domination and
boundary sharpness.

Panel~(b) of Figure~\ref{fig:theory} contrasts the finite-sample coverage of
the Hsu interval with that of Welch's approximation.  The Hsu coverage lies at
or above the nominal \(0.95\) for all \(\xi\), touching \(0.95\) only at the
boundary, while the Welch coverage drops below \(0.95\) near the boundary in
every design.  This is the precise sense in which the cylindrical
IM/Hsu PRS is exactly uniformly valid whereas Welch is only
approximately valid.

Finally, panel~(c) of Figure~\ref{fig:theory} illustrates
Theorem~\ref{thm:cond-shrink}.  Two explicit adaptive rules shrink the Hsu
multiplier.  Rule A shrinks it to \(0.9\,c_\alpha\) on the central part of the
\(R\)-distribution.  Rule B applies a uniform reduction of \(0.03\,c_\alpha\).
Both satisfy the positive-limsup hypothesis and consequently under-cover near the
least-favorable boundary (coverage falling to about \(0.933\) and \(0.946\)
respectively for \((5,15)\)), while the non-adaptive Hsu rule retains coverage
\(\geq0.95\).  This is the conditional-impossibility mechanism made concrete: a
genuine shrink region that remains visible under the least-favorable sequence
forces under-coverage.

\begin{figure}[htbp]
\centering
\includegraphics[width=\textwidth]{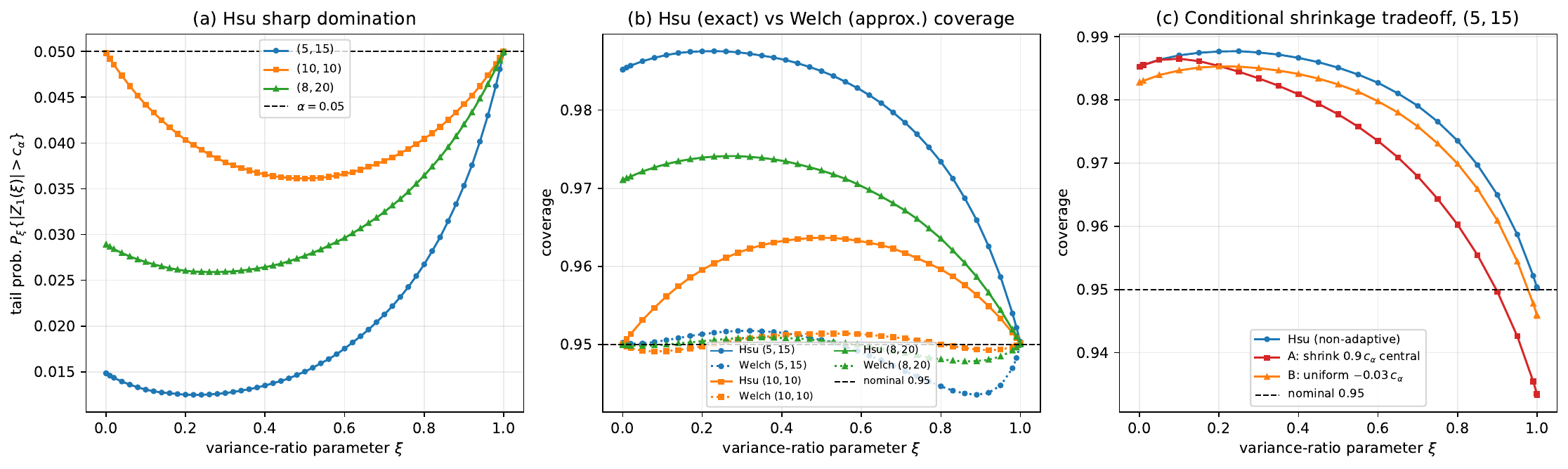}
\caption{Numerical validation of the theory, three designs
\((n_1,n_2)\in\{(5,15),(10,10),(8,20)\}\), nominal \(\alpha=0.05\).
\textbf{(a)} Hsu sharp domination: symmetric tail probability
\(P_\xi\{|Z_1(\xi)|>c_\alpha\}\) versus \(\xi\) at the Hsu critical value
\(c_\alpha=t_{1-\alpha/2,m}\); dashed line is the \(t_m\) bound \(=\alpha=0.05\).
The tail is strictly below \(\alpha\) in the interior and approaches \(\alpha\)
at the least-favorable boundary (\(\xi\to1\) for the unbalanced \((5,15)\) and
\((8,20)\) with \(m=n_1-1\); both boundaries for the balanced \((10,10)\) with
\(m=n_1-1=n_2-1\)).
\textbf{(b)} Uniform coverage of the two-sided interval for \(\psi\) versus
\(\xi\): Hsu (exact, solid) versus Welch (approximate, dotted).  Hsu holds at or
above the nominal \(0.95\) (dashed) uniformly in \(\xi\); Welch dips below
nominal near the boundary (minimum \(0.9436\), \(0.9491\), \(0.9478\) for the
three designs respectively).
\textbf{(c)} Conditional impossibility of uniform shrinkage, design \((5,15)\):
coverage versus \(\xi\) for the non-adaptive Hsu rule and two adaptive shrink
rules (A: shrink to \(0.9\,c_\alpha\) on central \(R\); B: uniform
\(-0.03\,c_\alpha\)); both adaptive rules satisfy the positive-limsup hypothesis
of Theorem~\ref{thm:cond-shrink} and dip below \(0.95\) near the boundary,
whereas Hsu does not.
Source: validated Monte Carlo, \(B=2\times10^{6}\) draws per grid point.}
\label{fig:theory}
\end{figure}

\subsection{Comparison with competing interval methods}
\label{sec:comparison}

The optimality theory of the preceding sections is a statement
about \emph{exact, finite-sample, prior-free} validity: the IM/Hsu
cylinder is the shortest interval whose coverage is at least \(1-\alpha\)
\emph{uniformly} in the variance ratio (Theorem~\ref{thm:cond-shrink} and the
shortest-projection theorem).  Competing two-sample procedures buy shorter
intervals in some regimes by \emph{adapting to the observed variance ratio}.  The
tradeoff principle predicts that any such method must pay for that adaptivity
either with conservatism elsewhere or with a loss of uniform validity.  This
section tests that prediction directly by Monte Carlo, comparing five intervals
for \(\psi=\mu_2-\mu_1\) at nominal level \(1-\alpha=0.95\):

\begin{enumerate}[label=(\roman*),itemsep=2pt]
\item \textbf{IM / Hsu--Scheffe}: \(\bar Y\pm t_{1-\alpha/2,m}\,f(S_1,S_2)\),
      \(m=\min(n_1,n_2)-1\), the cylindrical PRS of this paper, exactly and
      uniformly valid by Lemma~\ref{lem:hsu}.
\item \textbf{Welch--Satterthwaite} \citep{Welch1947,Satterthwaite1946}: the same
      pivot referred to a \(t_{\nu_W}\) with estimated degrees of freedom.
\item \textbf{Generalized fiducial / generalized confidence interval}
      \citep{Weerahandi1993}: the \(\alpha/2,\,1-\alpha/2\) quantiles of the
      generalized pivotal quantity
      \(R_\psi=\bar Y-\{(S_2/\sqrt{n_2})T_2-(S_1/\sqrt{n_1})T_1\}\),
      \(T_k\sim t_{n_k-1}\).
\item \textbf{Bayesian, independent Jeffreys priors} \citep{GhoshKim2001}: the
      central posterior credible interval for \(\mu_2-\mu_1\).
\item \textbf{Bootstrap-\(t\)} \citep{EfronTibshirani1993}: the studentized
      nonparametric bootstrap interval.
\end{enumerate}

\paragraph{An exact coincidence.}
For the mean difference, methods (iii), (iv), and Fisher's original fiducial
solution \citep{Fisher1935} are \emph{algebraically identical}: each reduces to
the law of \(\bar Y-\{(S_2/\sqrt{n_2})T_2-(S_1/\sqrt{n_1})T_1\}\) with
independent \(T_k\sim t_{n_k-1}\), which is the Behrens--Fisher distribution.  We
implemented (iii) and (iv) with independent random number streams as a check;
their estimated coverages agree to within Monte Carlo error at every grid point
(\(|\Delta|\le 0.008\), at most about three Monte Carlo standard errors of the
difference), confirming the identity numerically.
We therefore display them as a single curve, ``Fiducial \(=\) GFI \(=\)
Bayes(Jeffreys).''

\paragraph{Design.}
By location--scale invariance, the coverage of every method and its length
relative to \(\sigma_1\) depend only on \(\rho=\sigma_2^2/\sigma_1^2\) and on
\((n_1,n_2)\).  We therefore fix \(\mu_1=\mu_2=0\), \(\sigma_1=1\),
\(\sigma_2=\sqrt\rho\), so the true \(\psi=0\), and estimate coverage as the
frequency with which the interval contains \(0\).  We use
\(\rho\in\{10^{-2},\dots,10^{2}\}\) (\(25\) points, log-spaced) and
\((n_1,n_2)\in\{(3,3),(5,15),(10,10),(3,30),(5,30)\}\).  For Hsu, Welch and the
fiducial we use \(N=4\times10^{5}\) datasets per cell (Monte Carlo standard error
\(\approx 3.4\times10^{-4}\)); the fiducial band, being symmetric, is the law of
\(\bar Y\pm W\) with half-width \(W=(S_1/\sqrt{n_1})\,g(S_2\sqrt{n_1}/(S_1\sqrt{n_2}))\),
where \(g(r)\) is the \((1-\alpha/2)\)-quantile of \(|rT_2-T_1|\); tabulating
\(g\) once per design from \(10^{6}\) inner draws makes the high-resolution sweep
inexpensive.  The bootstrap-\(t\), which requires per-dataset resampling, is run
on a separate \(N=10^{4}\)/cell study with \(1{,}500\) resamples.  All random
number generators are seeded (grid seed \(2026060399\); bootstrap seed
\(99887766\)).

\paragraph{Results.}
Figure~\ref{fig:cmp-cov} shows coverage versus \(\rho\).  Welch dips below the
nominal \(0.95\) across a wide band, reaching \(0.923\) at \((n_1,n_2)=(3,30)\);
the bootstrap-\(t\) likewise under-covers, to about \(0.924\).  Neither carries a
finite-sample guarantee.  The Hsu interval never falls below nominal up to Monte
Carlo error: its smallest reading on the grid is \(0.9496\pm0.0003\), at the
least-favorable boundary, consistent with its \emph{proven} boundary value of
exactly \(0.95\) (the dedicated study of Section~\ref{sec:numerical},
\(B=2\times10^{6}\), pins the infimum at \(0.9500\)).  The fiducial \(=\) GFI
\(=\) Bayes family is \emph{conservative everywhere} (coverage in
\([0.975,\,0.997]\) on the grid), in agreement with the conservativeness result
of \citet{Robinson1976}.

The decisive comparison is therefore not coverage but length \emph{at the
variance ratios where validity binds}.  Figure~\ref{fig:cmp-len} shows interval
length relative to Hsu.  The fiducial does \emph{not} dominate Hsu; it
\emph{redistributes} width exactly as Theorem~\ref{thm:cond-shrink} dictates.  It
is shorter than Hsu only at the over-covered end of the \(\rho\)-axis, where Hsu
is already far above nominal and a shorter interval is of no value; at the
least-favorable boundary, where Hsu coverage is sharp at \(0.95\), the fiducial
is markedly \emph{longer}, by \(18\%\) for \((10,10)\), \(26\%\) for
\((5,15)\) and \((5,30)\), and \(44\)--\(45\%\) for \((3,30)\) and \((3,3)\)
(Figure~\ref{fig:cmp-len}, panel~(f)).  Thus the only competitor that retains
validity buys its occasional shortness precisely where shortness is worthless and
pays a large length premium precisely where it would matter.  Welch and the
bootstrap-\(t\) are shorter near the boundary, but only by abandoning coverage;
the bootstrap-\(t\) is additionally unstable in small samples, its mean length
exceeding \(30\times\) the Hsu length at \((3,3)\).

\begin{table}[htbp]
\centering
\caption{Summary of the comparison over the full \(\rho\times(n_1,n_2)\) grid
(nominal \(0.95\)).  ``Min.\ coverage'' is the smallest coverage over the grid.
``Length at LF boundary'' is the interval length relative to Hsu at the
least-favorable variance ratio, where Hsu coverage is sharp at \(0.95\), the
regime in which interval length is decision-relevant.  Only the Hsu interval
carries a constructive finite-sample uniform-validity proof.}
\label{tab:cmp}
{\small
\begin{tabularx}{\textwidth}{@{}l c >{\centering\arraybackslash}p{2.3cm} >{\raggedright\arraybackslash}X@{}}
\hline
Method & Min.\ coverage & Length at LF boundary & Finite-sample uniform validity \\
\hline
IM / Hsu--Scheffe              & \(0.950\)\,\(^{\dagger}\) & \(1.00\) & Yes, proven (Lemma~\ref{lem:hsu}) \\
Welch--Satterthwaite          & \(0.923\) & \(\approx 1.0\) & No, under-covers \\
Fiducial \(=\) GFI \(=\) Bayes & \(0.975\) & \(1.18\)--\(1.45\) & Conservative\(^{\ddagger}\); no constructive proof \\
Bootstrap-\(t\)                & \(0.924\) & unstable\(^{\S}\) & No, under-covers and unstable \\
\hline
\end{tabularx}
}

\smallskip
{\footnotesize \(^{\dagger}\)Grid minimum \(0.9496\pm0.0003\), consistent with
the proven boundary value \(0.95\); high-resolution infimum \(=0.9500\)
(Sec.~\ref{sec:numerical}).  \(^{\ddagger}\)Conservativeness of the
Behrens--Fisher fiducial is the result of \citet{Robinson1976}, specific to this
problem; it is not the prior-free constructive guarantee of the IM cylinder.
\(^{\S}\)The bootstrap-\(t\) mean length exceeds \(30\times\) the Hsu length at
\((3,3)\).}
\end{table}

\begin{figure}[htbp]
\centering
\includegraphics[width=\textwidth]{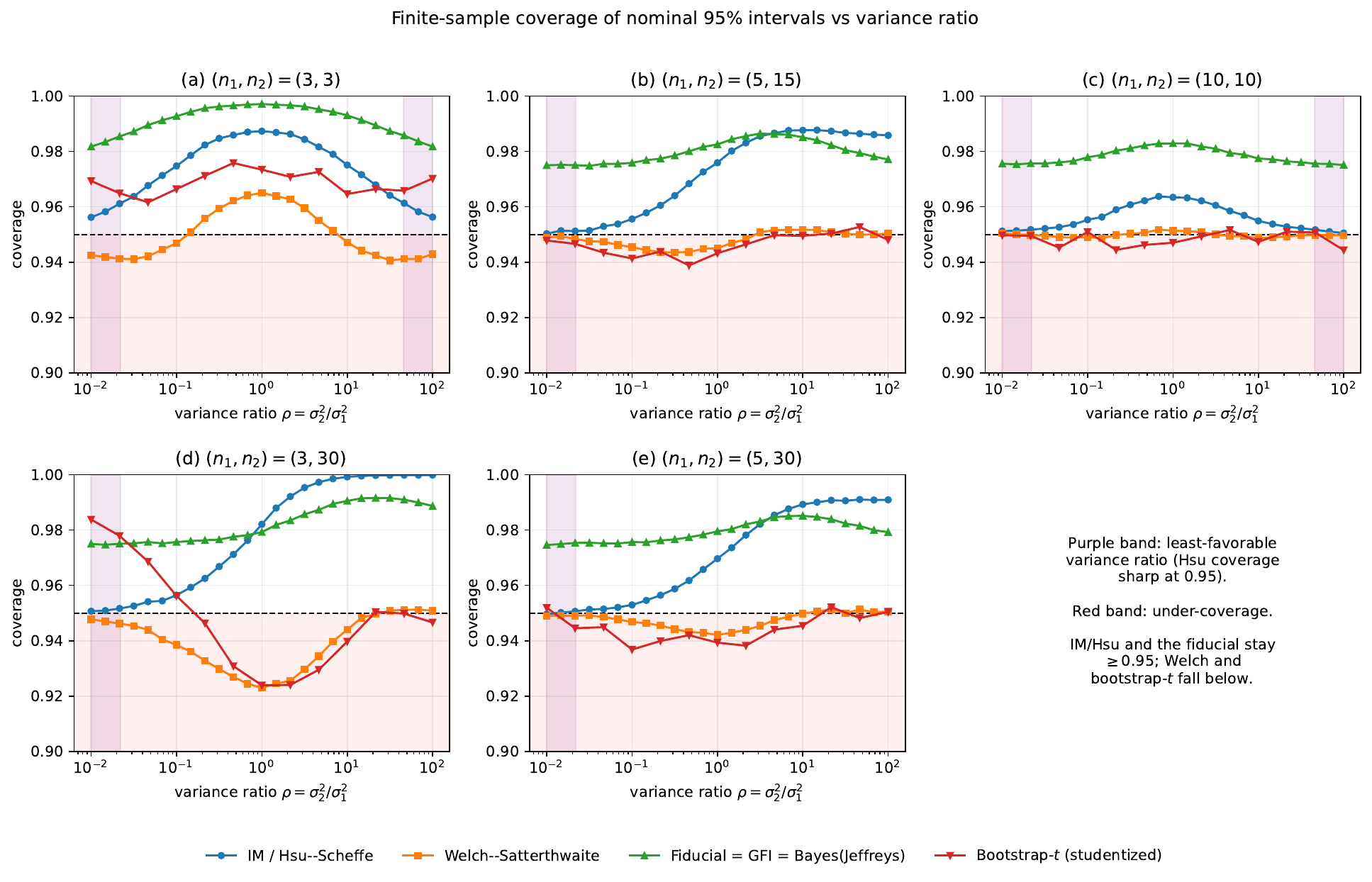}
\caption{Finite-sample coverage of nominal \(95\%\) intervals versus the
variance ratio \(\rho=\sigma_2^2/\sigma_1^2\), one panel per design (axis
zoomed to \([0.90,1.00]\)).  The red band marks under-coverage; the purple band
marks the least-favorable variance ratio, where Hsu coverage is sharp at
\(0.95\).  Welch and the bootstrap-\(t\) fall below nominal, markedly so for
small unbalanced designs; the Hsu interval stays \(\geq0.95\) (exactly \(0.95\)
at the boundary by Lemma~\ref{lem:hsu}); the fiducial \(=\) GFI \(=\) Bayes family
is conservative everywhere.  Hsu/Welch/fiducial: \(N=4\times10^{5}\) datasets per
cell (standard error \(\approx 3.4\times10^{-4}\)); bootstrap-\(t\):
\(N=10^{4}\).}
\label{fig:cmp-cov}
\end{figure}

\begin{figure}[htbp]
\centering
\includegraphics[width=\textwidth]{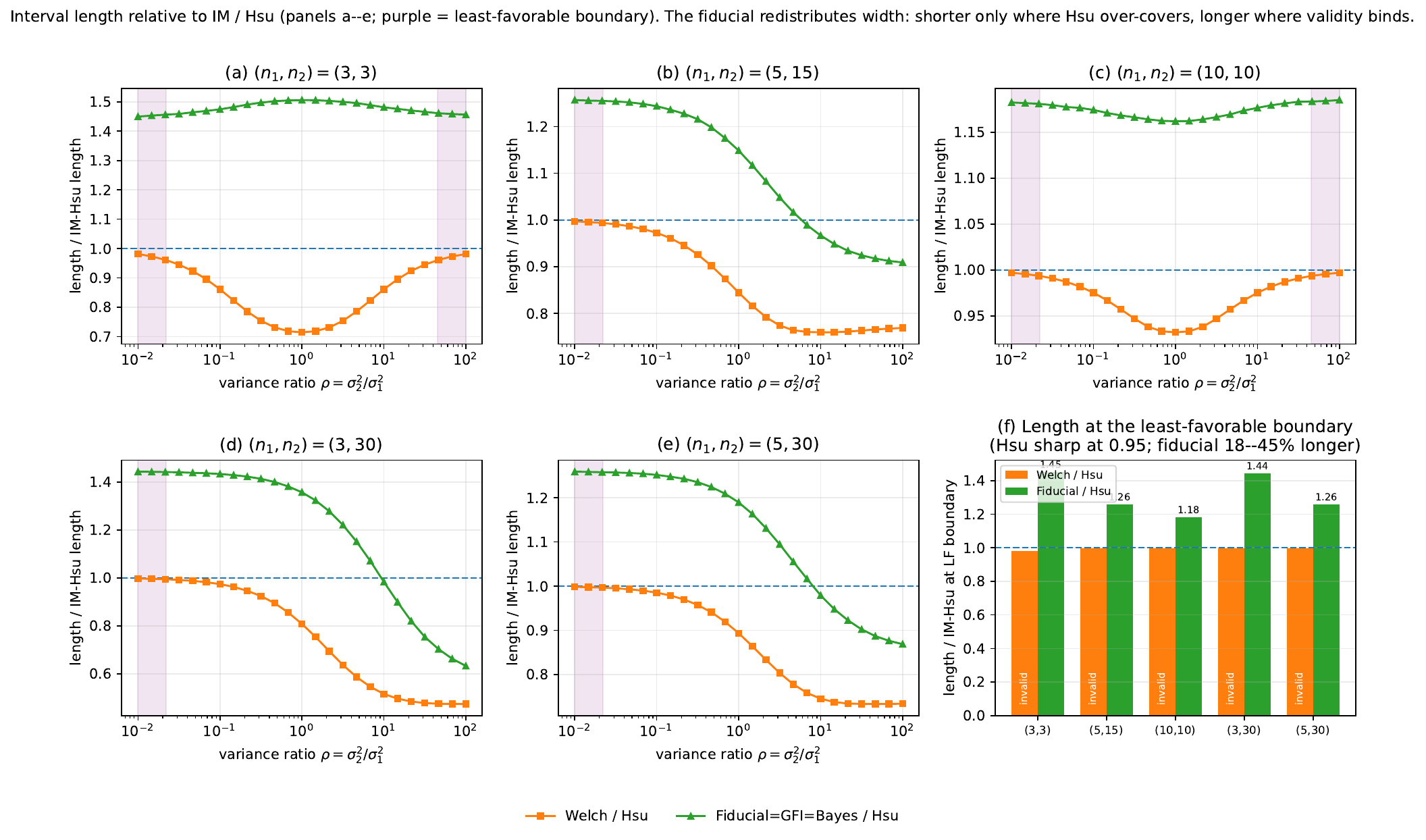}
\caption{Interval length relative to Hsu, and the redistribution of width.
Panels~(a)--(e): length of the Welch and fiducial intervals relative to Hsu
(dashed line) versus \(\rho\), one per design; the purple band marks the
least-favorable boundary.  The fiducial (which preserves validity) lies
\emph{above} the Hsu reference at the boundary and drops below it only at the
over-covered end; it redistributes width rather than uniformly shortening it.
Panel~(f): length relative to Hsu \emph{at the least-favorable boundary} for each
design; the fiducial is \(18\)--\(45\%\) longer than Hsu there, while Welch is no
shorter and is invalid (it under-covers, Fig.~\ref{fig:cmp-cov}).  This is the
empirical form of the no-uniform-shrinkage principle
(Theorem~\ref{thm:cond-shrink}): a genuinely adaptive, validity-preserving
procedure cannot beat the Hsu length where validity binds.}
\label{fig:cmp-len}
\end{figure}

\paragraph{Interpretation.}
The comparison confirms the theory rather than merely illustrating it.  Welch and
the bootstrap-\(t\) achieve shorter intervals by exploiting the observed variance
ratio, and they pay for it exactly where Theorem~\ref{thm:cond-shrink} says they
must: under-coverage that is largest at small, unbalanced designs and near the
least-favorable boundary.  The generalized fiducial is the more instructive
competitor, because it \emph{is} a genuinely two-dimensional, adaptive procedure
and it \emph{does} retain validity (conservatively, by \citet{Robinson1976}).
Precisely for that reason it cannot escape the tradeoff principle: it redistributes
interval width across variance-ratio regimes, shortening the interval only where
Hsu already over-covers and lengthening it by \(18\)--\(45\%\) at the
least-favorable boundary, where Hsu is sharp.  It therefore offers no uniform
improvement over Hsu, and a strict deterioration in the one regime that governs
worst-case validity, a concrete instance of the shortest-projection and
no-uniform-shrinkage theorems.  Crucially, only the Hsu interval's validity rests
on a \emph{constructive, prior-free, finite-sample} argument (Hsu's domination);
Welch and the bootstrap-\(t\) are approximate, and the fiducial relies on a
problem-specific conservativeness theorem.  Any apparent average efficiency of an
adaptive method, moreover, is realized only by averaging length over
variance-ratio regimes, precisely the prior- or weight-dependent comparison the
Discussion cautions against treating as intrinsic.

\section{Conclusion}
\label{sec:conclusion}

The Behrens--Fisher problem remains important because it exposes the limits of
automatic nuisance-parameter elimination, and the IM framework makes this point
especially sharp.  After conditioning and regular marginalization, the exact
reduced association is not one-dimensional but two-dimensional, carrying a
mean-contrast coordinate and a variance-ratio coordinate.  Martin and Liu's
one-dimensional generalized marginal IM is therefore most naturally read as a
cylindrical two-dimensional predictive random set: sharp in its mean-contrast
projection, and vacuous in the second coordinate, predicting the variance ratio
only by its entire range.

Within this geometry the IM/Hsu solution is optimal in a precise validity-first
sense.  Among prior-free procedures that retain exact, uniform, finite-sample
validity, it attains the shortest plausibility interval; we established the
accompanying minimaxity and admissibility in the cylindrical class and, via
symmetric unimodality, the shortest uniformly valid projection among all
two-dimensional sets judged through their symmetric projection.  The companion
tradeoff principle shows the price of this validity: no genuinely adaptive
procedure can uniformly shorten the interval, only redistribute its width across
variance-ratio regimes.  The Monte Carlo study confirms both messages.  The Hsu
interval holds coverage at or above the nominal level uniformly, whereas Welch and
the studentized bootstrap dip below it; and the generalized fiducial (which
coincides with the Jeffreys--Bayes and Fisher fiducial interval for the mean
difference), though conservative, does not dominate Hsu, being shorter only where
Hsu already over-covers and \(18\)--\(45\%\) longer at the least-favorable
boundary, where validity binds.

These comparisons should not be collapsed into a single average over the unknown
variance ratio: such an average is prior- or weight-dependent, not an intrinsic
prior-free efficiency comparison.  Two questions remain open.  The first is
constructive: whether a calibrated noncylindrical two-dimensional predictive
random set can use the variance-ratio coordinate to redistribute interval width
usefully while preserving exact uniform validity.  Our results bound what any such
construction can achieve, and they lead us to a stronger belief, which we record
as a conjecture.

\begin{conjecture}[Prior-free optimality]
\label{conj:prior-free}
The symmetric-interval-projection restriction in the shortest uniformly valid
projection theorem and in Corollary~\ref{cor:proj-admiss} is inessential.  Among
all prior-free procedures that are exactly and uniformly valid for \(\psi\) over
the variance ratio, none can have plausibility interval length no greater than
that of the IM/Hsu interval at every \(\xi\) and strictly smaller at some
\(\xi\).  Equivalently, the IM/Hsu interval is admissible in length, uniformly in
the variance ratio, within the entire prior-free class.
\end{conjecture}

If this holds, the conservative Hsu--Scheffe interval is not merely a safe default
but the unique prior-free endpoint of the efficiency frontier: no genuinely
adaptive construction could improve on it at any variance ratio without
surrendering either exact validity or its prior-free character.

The second question is Bayesian, and here we venture no conjecture, only an
encouraging hint.  The Jeffreys-prior interval examined above is conservative and
coincides with the fiducial.  For the balanced designs, it over-covers throughout
(Figure~\ref{fig:cmp-cov}), while the IM/Hsu interval, also valid, is uniformly
shorter (Figure~\ref{fig:cmp-len}); the Bayesian solution thus seems to be
dominated by IM/Hsu in those cases.  This suggests that the Jeffreys interval may be inadmissible here, and raises the
question of whether it can be improved from within the Bayesian formalism: whether
some other prior, or class of priors, yields valid credible intervals that are
shorter where the Jeffreys interval is wasteful.  We leave this fuller theoretical
investigation to be carried out elsewhere, as an inspiring direction rather than a
claim.

\end{document}